\documentclass{amsproc}
\usepackage{amssymb}

\newtheorem{theorem}{Theorem}[section]
\newtheorem{lemma}[theorem]{Lemma}

\newtheorem{proposition}[theorem]{Proposition}
\newtheorem{corollary}[theorem]{Corollary}

\theoremstyle{definition}
\newtheorem{definition}[theorem]{Definition}

\theoremstyle{remark}\newtheorem*{claim}{Claim} 
\theoremstyle{remark} 
\theoremstyle{remark}
\newcommand{\N}{\mathbb{N}}
\newcommand{\R}{\mathbb{R}}

\numberwithin{equation}{section}

\begin{document}
\title[On complemented copies of $c_0(\omega_1)$ in $C(K^n)$ spaces] {On complemented copies of $c_0(\omega_1)$ in $C(K^n)$ spaces}


\author{Leandro Candido}
\address{University of S\~ao Paulo, Department of Mathematics, IME, Rua do Mat\~ao 1010,  S\~ao Paulo, Brazil}
\email{lcandido@impan.pl}
\thanks{The research of the first author was conducted while he was a postdoctoral
fellow at the Institute of Mathematics, Polish Academy of Sciences supported by FAPESP, process number 2013/20703-6.}

\author{Piotr Koszmider}
\address{Institute of Mathematics, Polish Academy of Sciences,
ul. \'Sniadeckich 8,  00-656 Warszawa, Poland}
\email{\texttt{piotr.koszmider@impan.pl}}
\thanks{The second author was partially supported by  grant
PVE Ci\^encia sem Fronteiras - CNPq, process number 406239/2013-4 } 




\begin{abstract} 
Given a compact Hausdorff space $K$ we consider the Banach space
of real continuous functions  $C(K^n)$ or
equivalently the $n$-fold  injective tensor product $\hat\bigotimes_{\varepsilon}C(K)$
or  the Banach space of vector valued continuous functions $C(K, C(K, C(K ..., C(K)...)$.
We address the question of the existence of complemented copies of $c_0(\omega_1)$
in $\hat\bigotimes_{\varepsilon}C(K)$ under the hypothesis that $C(K)$ contains an isomorphic copy of $c_0(\omega_1)$.
This is related to the results of
 E. Saab and P. Saab that $X\hat\otimes_\varepsilon Y$ contains a complemented copy of $c_0$,
if one of the infinite dimensional Banach spaces $X$ or $Y$ contains  
 a copy of $c_0$ 
and of  
 E. M. Galego and J. Hagler   that it follows from Martin's Maximum
 that if $C(K)$ has density $\omega_1$ and contains  a copy of 
$c_0(\omega_1)$, then $C(K\times K)$ contains a complemented copy $c_0(\omega_1)$.

The main result is that under the assumption of $\clubsuit$ for every $n\in \N$
there is a compact Hausdorff space $K_n$ of weight $\omega_1$ such that
$C(K)$ is
 Lindel\"of in the weak topology, $C(K_n)$ contains a copy
of $c_0(\omega_1)$,  $C(K_n^n)$ does not contain a complemented copy
of $c_0(\omega_1)$ while $C(K_n^{n+1})$ does  contain a complemented copy
of $c_0(\omega_1)$.
This shows that additional set-theoretic assumptions in Galego and Hagler's
nonseparable version of Cembrano and Freniche's theorem are necessary as well as
clarifies in the negative direction the matter unsettled in a paper of
 Dow, Junnila and  Pelant  whether  half-pcc  Banach spaces must be weakly pcc.

\end{abstract}

\maketitle


\section{Introduction}
Given a compact Hausdorff space $K$ the geometry of the Banach space $C(K\times K)$, besides
being interesting by itself, is additionally important
because $C(K\times K)$s form paradigmatic examples of Banach spaces of vector valued continuous functions
$C(K, C(K))$ and of the injective tensor products $C(K)\hat\otimes_\varepsilon C(K)$
and hence they are relevant to the investigations of the properties of these tensor products
$X\hat\otimes_\varepsilon Y$ in terms of the properties of $X$ and $Y$.
It is well known, by a surprising and celebrated result of P. Cembranos (\cite{Cem})
and F. Freniche (\cite{freniche}), that if $C(K)$ contains a copy of $c_0$ (i.e.,  $C(K)$
is infinite dimensional), then $C(K\times K)$ always contains a complemented copy of $c_0$. 
This result has been generalized by E. Saab and P. Saab (\cite{saabsaab}) to any tensor product
$X\hat\otimes_\varepsilon Y$ of infinite dimensional Banach spaces $X, Y$ where
one of them contains $c_0$ or even to spaces of compact operators
(\cite{ryan}).

A consistent nonseparable version of this result has been recently obtained
 by E. M. Galego and J. Hagler (\cite[Corollary 4.7]{GaHa}) who,
in particular,  proved that it is relatively consistent with ZFC
 that if $C(K)$ has density $\omega_1$ and $C(K)$ has a copy of $c_0(\omega_1)$, 
then $C(K\times K)$ has a complemented copy $c_0(\omega_1)$. 
Their proof relies on  Todorcevic's analysis of nonseparable Banach spaces 
(\cite[Corollary 6]{Todo}) under the assumption of
an additional set-theoretic axiom known as Martin's Maximum (\cite{ForMagShe}).
However A. Dow, H. Junnila and J. Pelant constructed in ZFC a Banach space
(of the form $C(K)$) of density $2^\omega$ which contains a copy of $c_0(\omega_1)$
but such that its injective tensor square has no complemented copy $c_0(\omega_1)$ (Example 2.16 of
\cite{DowJunPel}).
For more on complemented copies of $c_0(\omega_1)$ see \cite{argyros} and \cite{KosZie}.

In order to discuss our results, let us recall some terminology.
A Banach space $X$ is called {\sl weakly pcc} if and only if
any point-finite family of open sets in the weak topology on $X$ is countable and is called
{\sl half-pcc} if any point-finite family of half spaces (i.e., sets of the form
$\{x\in X: x^*(x)>a\}$ for some $a\in \R$ and $x^*\in X^*$) is countable.
This kind of chain conditions were first considered by Rosenthal in \cite{rosenthal} and
 are fundamental in the work of Dow, Junnila and Pelant (\cite{DowJunPel}).
If $X$ is a Banach space, then a sequence $(x_\xi^*)_{\xi<\omega_1}$ of elements
of the unit sphere of $X^*$ is called an {\sl $\omega_1$-Josefson-Nissenzweig sequence}
if and only if $(x_\xi^*(x))_{\xi<\omega_1}$ belongs to $c_0(\omega_1)$
for any $x\in X$. This notion plays a crucial role in the
result of Galego and Hagler and of Todorcevic mentioned above (\cite{GaHa, Todo}).

\begin{theorem}[\cite{DowJunPel, GaHa}]\label{introequiv} Let $X$ be a Banach space. The following conditions are equivalent:
\begin{enumerate}
\item There is a linear bounded operator $T:X\rightarrow c_0(\omega_1)$
which has non-separable range,
\item There is an $\omega_1$-Josefson-Nissenzweig sequence in $X$,
\item $X$ is not half-pcc.
\end{enumerate}
\end{theorem}
 \begin{proof}
If $\phi_\alpha=T^*(\delta_\alpha)/\|T^*(\delta_\alpha)\|$, where $\delta_\alpha(f)=f(\alpha)$
for each $\alpha<\omega_1$ and $f\in c_0(\omega_1)$, we obtain that
$\phi_\alpha(x)=T(x)(\alpha)/\|T^*(\delta_\alpha)\|$. But for $\alpha$s from
an uncountable set $A\subseteq \omega_1$ the numbers $\|T^*(\delta_\alpha)\|$
are uniformly separated from zero. Hence $(\phi_\alpha)_{\alpha\in A}$ is
$\omega_1$-Josefson-Nissenzweig sequence in $X$. 

Given an $\omega_1$-Josefson-Nissenzweig sequence $(\phi_\alpha)_{\alpha<\omega_1}$
 in $X$ define an operator $T:X\rightarrow c_0(\omega_1)$ by
$T(x)=(\phi_\alpha(x))_{\alpha<\omega_1}$. By the Hahn-Banach theorem it has a nonseparable range.
The equivalence of (1) and (3) is proved in Theorem 1.6. of \cite{DowJunPel}.
\end{proof}

Note that the conditions above are in general much weaker than $C(K)$  containing a complemented 
copy of $c_0(\omega_1)$, but following the ideas of Cembranos and Freniche,
Galego and Hagler showed that if $X$ contains a 
copy of $c_0(\omega_1)$, then $C(K\times K)$ contains a complemented copy of 
$c_0(\omega_1)$ under one of the above conditions from \ref{introequiv}.
Actually, Todorcevic proved that under Martin's Maximum 
any nonseparable Banach space of density $\omega_1$  satisfies the above conditions, what gives 
the previously mentioned result of Galego and Hagler under Martin's Maximum.
The separable version of Todorcevic's result is the classical theorem of Josefson and Niessenzweig (\cite{josefson, nissenzweig})
which implies that on any infinite dimensional Banach space there is a linear operator with 
its range dense in $c_0$.

One of the goals of the research project leading to this
 paper was to decide if Galego and Hagler's nonseparable version of Cembranos  and Freniche's theorem 
and consequently the Saabs' theorem, is indeed  sensitive to 
its additional set-theoretic assumption, i.e., that one cannot obtain the same result without
making it.
We confirmed it by proving that  it is consistent that 
there exist a compact Hausdorff space $K$ such that the
density of the Banach space $C(K)$ is $\omega_1$ and $C(K)$ contains a copy
of $c_0(\omega_1)$ but $C(K\times K)$ does not 
contain a complemented copy of $c_0(\omega_1)$.
In fact we encountered two different kinds of examples of spaces satisfying the above statement.
We found one of them existing already in the literature (2.17 of \cite{DowJunPel}
discussed in Section 4 and denoted there $DJP_1$) 
after we have constructed our original space. The analysis of the differences between
these examples led us to a more delicate result:

\begin{theorem}\label{main} It is consistent that there are compact Hausdorff spaces 
$K_n$ for all $1\leq n\leq\omega$ such that all $C(K_n)$s contain a copy of  $c_0(\omega_1)$ and
$C(K^m_n)$ contains a  complemented copy of $c_0(\omega_1)$ for
$n\leq \omega$ and $m<\omega$  if and only if
$n<m$.
\end{theorem}
\begin{proof}
$K_n$ is the  space constructed in Section 3 whose properties
are stated in Proposition \ref{compactclub}.  Since $K_n$ is $(n+1)$-diverse 
(see Definition \ref{diversespace}), by Proposition \ref{diversepcc}
 the space $C(K_n^n)$ is half-pcc and so by Theorem \ref{introequiv} $C(K_n^n)$ does not contain a complemented
copy of $c_0(\omega_1)$. On  the other hand by Proposition \ref{compactclub}
  there is an $n$-to-$1$ continuous map from 
$K_n\setminus\{\infty\}$ onto $[0,\omega_1)$ with the order topology, so $C(K_n^{n+1})$
contains a complemented copy of $c_0(\omega_1)$ by Proposition \ref{exists-complemented}.

$K_\omega$ is the  example from \ref{komega} or 
if we are not interested in  additional properties, which we state in Theorem
\ref{theoremlindelof}, the Example 2.17 of \cite{DowJunPel} which we call
$DJP_1$ and analyze in Section 4.  Since these are nonseparable scattered
compact spaces, they contain uncountable subspaces of isolated points, and so
$C(K_\omega)$ contains $c_0(\omega_1)$.  The properties of $K_\omega$
follow from Proposition \ref{komega} and Theorem \ref{weaklypccndiverse}.
The properties of $DJP_1$ follows from Corollary \ref{powerfullydjp} and 
Theorem \ref{introequiv}.
\end{proof}

For additional properties of the $K_n$s see Theorem \ref{theoremlindelof}.
In fact our constructions are generalizations of the space from \cite{KosZie} which can
serve above as $K_1$.
As we note in Corollary (\ref{powerfullywpcc}) that for $K$ compact scattered  if $C(K)$ is weakly pcc, then
$C(K^n)$ is weakly pcc for every $n$, our examples $K_n$ for $n<\omega$ from Theorem \ref{main}
give the following:

\begin{theorem} It is consistent that there are half-pcc spaces $C(K)$
of density $\omega_1$ which are not weakly pcc.
\end{theorem}

This seem to be unknown until now (see page 1330 of  \cite{DowJunPel}), but we still do
not know if such examples can be constructed without any additional assumptions (see \ref{uselessdjp2}).
Among the three equivalent conditions of Theorem \ref{introequiv}  the topological notion of being half-pcc is the simplest.
Moreover in the case of scattered compact spaces it allows a simple topological criterion 
extracted in \cite{DowJunPel} from a paper of Arhangelskii and Tka\v cuk \cite{Arh}. 
Our refinements of these notions are the following:

\begin{definition}\label{diversepoint}
Let $K$ be a compact space, $m \in \mathbb{N}$ and let $F_1,\ldots,F_k$ a partition of $\{1,\ldots,m\}$. A point $(x_1,\ldots,x_{m})\in K^m$ is said to be \emph{$(F_1,\ldots,F_k)$-diverse} if $\{x_j: j\in F_i\}\cap \{x_j:j\not\in F_i\}=\emptyset$ for all $1\leq i \leq k$. 
\end{definition}

\begin{definition}\label{diversespace}
Let $K$ be a Hausdorff compact and $n\in \mathbb{N}$. We  say that $K$ is \emph{$n$-diverse} if for any given $m\in \N$ and for any partition $F_1,\ldots,F_k$ of $\{1,\ldots,m\}$ with $k \leq n$, any sequence $\{(x_1^{\alpha},\ldots,x^{\alpha}_{m})\}_{\alpha<\omega_1}\subseteq K^{m}$ of $(F_1,\ldots,F_k)$-diverse points admits a cluster point which is $(F_1,\ldots,F_k)$-diverse. 
\end{definition}

\begin{theorem}\label{weaklypccndiverse} Suppose that $K$ is a scattered compact space. 
$C(K)$ is weakly pcc if and only if $K$ is $n$-diverse for each $n\in \N$
\end{theorem}
\begin{proof} By Proposition 2.2. and Lemma 2.13 of \cite{DowJunPel}
we need to prove that $K^n\setminus\Delta_n$ is $\omega_1$ compact for all $n\in \N$
if and only if $K$ is $n$-diverse for all $n\in \N$.  For the forward
implication given a sequence of $v_\xi=(x^1_\xi, ..., x^m_\xi)$ of $(F_1,\ldots F_k)$-diverse points in $K^m$
one can assume that there is a partition of $\{1,..., m\}$ into sets
$(A_k)_{k\leq l}$ for some $l\leq m$ such that the coordinates of the points $v_\xi$
 in the same part of the partition are equal.
 Form points $w_\xi$s of $K^l\setminus \Delta_l$ from 
the coordinates of $v_\xi$s in distinct $A_k$s. Use the $\omega_1$-compactness
to obtain an accumulation point of $w_\xi$s
in $K^l\setminus \Delta_l$ and use it to
form an $(F_1, ...,F_k)$-diverse accumulation point of $v_\xi$s. The backward implication is clear.
\end{proof}

We obtain the following:

\begin{theorem}\label{diverseintro}
Let $K$ be a compact, totally disconnected space and $n\in\N$.
Each of the following conditions implies the next.
\begin{enumerate} 
\item $K$ is $(n+1)$-diverse,
\item $K^n$ is half-pcc,
\item $C(K^n)$ contains no complemented  isomorphic copy of $c_0(\omega_1)$,
\item There is no point $\infty\in K$ such that $K\setminus\{\infty\}$
can be  mapped onto $[0,\omega_1)$ by an $(n-1)$-to-$1$ continuous map.
\end{enumerate}
\end{theorem}
\begin{proof} (1) implies  (2) by Proposition \ref{diversepcc}.  (2) implies  (3) by Theorem \ref{introequiv}.
Proposition \ref{exists-complemented} yields the implication from (3) to (4).
\end{proof}
Our  examples $K_n$ for $n\leq \omega$, unlike the example 2.17 of \cite{DowJunPel}, have all these properties for a given $n\in \N$ and
none of these properties for bigger numbers. Another distinct feature of our examples is
the Lindel\"of property in the weak topology:

\begin{theorem}\label{theoremlindelof} It is consistent that
there are totally disconnected compact Hausdorff spaces $K_n$ satisfying
Theorem \ref{main} such that $C(K_n)$ is
Lindel\"of in the weak topology and 
$K$ is $(n+1)$-diverse and there is a point $\infty\in K$ such that $K\setminus\{\infty\}$
can be  mapped onto $[0,\omega_1)$ by an $n$-to-$1$ continuous map.
\end{theorem}
\begin{proof} Applying Proposition \ref{compactclub} and Theorem \ref{diverseintro} we are left only with proving the
Lindel\"of property of $C(K_n)$ with the weak topology. 
First note that the Lindel\"of property in the weak topology
in $C(K)$ for $K$ scattered is equivalent to this property for the pointwise converegence topology.
Now the Lindel\"of property follows 
from the fact that $K_n^{(\omega_1)}=\emptyset$   by a theorem of
G. Sokolov (Theorem 2.3. of \cite{sokolov}) which says that  for
such scattered spaces the Lindel\"of property in $C_p(K)$ is equivalent to
$\aleph_0$-monolithicity of $K$ that is the property that closures of countable sets 
have countable networks. In our case the closures of
countable sets are included in sets
of the form $[0,\beta]\cup\{\omega_1\}$ 
for some countable ordinal $\beta$ (see Proposition \ref{compactclub}) 
and so are metrizable because they are countable and scattered.
\end{proof}

The Lindel\"of property in the weak topology is relevant here because
 it is proved in Proposition 1.16 of  \cite{DowJunPel} that a nonseparable
weakly Lindel\"of determined $C(K)$ cannot be half-pcc.
Also, as noted at the begining of Section 1 of  \cite{DowJunPel} the Lindel\"of property
in the case of  the weak topology is equivalent to being paracompact.
Thus our spaces $C(K_n)$ are paracompact in the weak topology
in contrast to  the space $C(DJP_1)$ from  \cite{DowJunPel} (see \ref{powerfullydjp}) which is
even not $\sigma$-metacompact with the weak topology since it is weakly pcc.

We obtain our consistent examples assuming the combinatorial
principle $\clubsuit$ of Ostaszewski (see \cite{Osta}, it is explained at the begining of
Section 2.) and we note that this principle  is
also sufficient to obtain the construction 2.17  from \cite{DowJunPel} originally
obtained from $\diamondsuit$ of R. Jensen (\cite{Kunen}).
The advantage of
$\clubsuit$ over $\diamondsuit$ is that the first principle is compatible with both
CH and its negation while $\diamondsuit$ implies CH.
This kind of examples cannot be obtained without additional set-theoretic
assumptions because under Martin's Maximum Banach spaces of density
$\omega_1$ map continuously and linearly into $c_0(\omega_1)$ with nonseparable ranges
(\cite{Todo}) and under the P-ideal dichotomy any  weakly Lindel\"of 
$C(K)$  space with $K^{(\omega_1)}=\emptyset$ containing an isomorphic
copy of $c_0(\omega_1)$ contains a complemented copy
of $c_0(\omega_1)$ (Theorem 3.2 of \cite{KosZie}). In this context, because $DJP_2$ (see section 4)
is a subspace
of separable scattered space, one should also recall 
 a result of R. Pol saying that if $K$ is separable, nonmetrizable, scattered of countable height,
then $C(K)$ is not Lindel\"of in the weak topology (Theorem 2 of \cite{pol}, see also \cite{dowsimon}). 
But we do
not know if dropping the requirement of $C(K)$ being Lindel\"of
in the weak topology one can construct in ZFC a $K$ of weight
continuum such that $C(K^n)$ contains complemented copies of
$c_0(\omega_)$  for some $n$s and not for others. Using
the notions of $n$-diverse spaces in the context of the ZFC example
2.16 of \cite{DowJunPel} called here  $DJP_2$ seem not to work
as in the case of $DJP_1$ which is shown in propositions \ref{djp2} and
\ref{uselessdjp2}.

We will denote by $\mathcal{L}(\omega_1)$ the set of all countable ordinals which are limit ordinals. 
The other notation is standard.

\section{$n$-to-1 maps onto $[0,\omega_1)$, $n$-diverse spaces and the pcc}

\begin{proposition}\label{exists-complemented}
Let $K$ be compact totally disconnected space and $\infty\in K$.
If there exists a continous surjective map $\phi:K\setminus\{\infty\} \to [0,\omega_1)$ 
such that $|\phi^{-1}[\{\alpha\}]|\leq n$ for all $\alpha<\omega_1$ and some $n\in \N$, where $[0,\omega_1)$ is
 endowed with the order topology, then 
$C(K^{n+1})$ contains a complemented copy of $c_0(\omega_1)$. In particular
$C(K^{n+1})$ is not half-pcc.

\end{proposition}
\begin{proof}
Put $L=K\setminus\{\infty\}$. For each $\gamma\in \omega_1$ pick any $x_\gamma\in \phi^{-1}[\{\gamma\}]$.
We will consider the sets of points
$x_{\gamma+1},  \ldots, x_{\gamma+(n+2)}$ for $\gamma\in \mathcal{L}(\omega_1)$. 
All these points are isolated in $L$ and hence in $K$ because successors are isolated in $[0,\omega_1)$
and $\phi$ is continuous.
Note that given  pairwise disjoint clopen partition    of
$K$ into sets $U_1,  \ldots, U_k$ for some $k\in \N$ the set of  all 
$\gamma\in \mathcal{L}(\omega_1)$ such that the sets $U_1, \ldots, U_k$ separate
all the points $x_{\gamma+1}, \ldots, x_{\gamma+(n+2)}$ is at most finite.
Indeed, at most one of the clopen sets, say $U_j$ for some $1\leq j\leq k$
contains $\infty$. So, any accumulation point
of an infinite sequence $x_{\gamma_m+i}$ from outside of $U_j$ and $1\leq i\leq n+2$ for
an increasing sequence $(\gamma_m)_{m\in \N}$  must be in 
$\phi^{-1}[\{\sup_{m\in \N} \gamma_m\}]$. This set has at most $n$ elements but there
will be $n+1$ such distinct accumulation points, if $U_1, \ldots, U_k$ separate
all the points $x_{\gamma_m+1}, \ldots, x_{\gamma_m+(n+2)}$ for infinitely many $m\in \N$, a contradiction.

Now we define measures on $K^{n+1}$ which will serve for defining a projection from $C(K^{n+1})$
onto a copy of $c_0(\omega_1)$. Let $S(k)$ denote the set of all permutations
of $\{1,  \ldots, k\}$ and let $\mathrm{sgn}(\sigma)$ denote the sign of a permutation $\sigma\in S(k)$ for some $k\in \N$.
Let $\lambda_{\gamma}$ be  defined by 
$$\lambda_{\gamma}=\sum_{\sigma\in S(n+2)}\mathrm{sgn}(\sigma)\cdot
 \delta_{\{(x_{\gamma+\sigma(1)}, \ldots, x_{\gamma+\sigma(n+1)})\}}.$$ 
Note that $\sigma$ above is a permutation of all $n+2$ points but the coordinates of the point 
of $K^{n+1}$ use only the first $n+1$ numbers.

Next note that for any clopen $V_1, \ldots, V_{n+1}$ in $K$
the set $\{\gamma \in \mathcal{L}(\omega_1):\lambda_{\gamma}(V_1\times  \ldots\times V_{n+1})\neq 0\}$ is finite.
To see this, first note that by considering all Boolean components $U_1, \ldots, U_k$
for some $k\in \N$ of the finite field of sets generated by
$V_1,  \ldots, V_{n+1}$ we know that for only finitely many $\gamma$s the sets $U_1,  \ldots, U_k$ and so $V_1,  \ldots, V_{n+1}$ 
separate all the points $\{x_{\gamma+1}, \ldots, x_{\gamma+(n+2)}\}$. 

For those $\gamma$ where none of the
$V_1,  \ldots, V_{n+1}$s  separate say $x_{\gamma+i}$ form $x_{\gamma+j}$ and those
 $\sigma\in S(n+2)$  that $\sigma(i), \sigma(j)<n+2$ we have that
$$(x_{\gamma+\sigma(1)},  \ldots,x_{\gamma+\sigma(i)},  \ldots, x_{\gamma+\sigma(j)},  \ldots, x_{\gamma+\sigma(n+1)})
\in V_1\times  \ldots\times V_{n+1}$$
if and only if 
$$(x_{\gamma+\sigma'(1)},  \ldots,x_{\gamma+\sigma'(i)},  \ldots, x_{\gamma+\sigma'(j)},  \ldots, x_{\gamma+\sigma'(n+1)})
\in V_1\times  \ldots\times V_{n+1}$$
where $\sigma'$ is obtained by composing $\sigma$ with the transposition of $i$ and $j$.
In this case $\mathrm{sgn}(\sigma)=-\mathrm{sgn}(\sigma')$ and so the contributions
of $\delta_{\{(x_{\gamma+\sigma(1)},  \ldots, x_{\gamma+\sigma(n+1)})\}}$ and
$\delta_{\{(x_{\gamma+\sigma'(1)},  \ldots, x_{\gamma+\sigma'(n+1)})\}}$ 
on $V_1\times  \ldots\times V_{n+1}$ cancel each other. 

For those $\gamma$ where none of the
$V_1,  \ldots, V_{n+1}$s  separate say $x_{\gamma+i}$ form $x_{\gamma+j}$ and those
 $\sigma\in S(n+2)$  that $ \sigma(j)=n+2$ we have that
$$(x_{\gamma+\sigma(1)},  \ldots,x_{\gamma+\sigma(i)},  \ldots, x_{\gamma+\sigma(n+1)})
\in V_1\times  \ldots\times V_{n+1}$$
if and only if 
$$(x_{\gamma+\sigma'(1)},  \ldots,x_{\gamma+\sigma'(j)},  \ldots, x_{\gamma+\sigma'(n+1)})
\in V_1\times  \ldots\times V_{n+1}$$
where $\sigma'$ is obtained by composing $\sigma$ with the transposition of $i$ and $j$,
i.e., $\sigma'(i)=n+2$.
In this case $\mathrm{sgn}(\sigma)=-\mathrm{sgn}(\sigma')$ and so the contributions
of $\delta_{\{(x_{\gamma+\sigma(1)},  \ldots, x_{\gamma+\sigma(n+1)})\}}$ and
$\delta_{\{(x_{\gamma+\sigma'(1)},  \ldots, x_{\gamma+\sigma'(n+1)})\}}$ 
on $V_1\times  \ldots\times V_{n+1}$ cancel each other as well. 

Having fixed distinct $1\leq i, j\leq n+2$ such that the sets $V_1\times  \ldots\times V_{n+1}$
do not separate the points $x_{\gamma+i}$ from $x_{\gamma+j}$
for a fixed $\gamma<\omega_1$
 it is clear that the set of all permutations can be partitioned into two sets such that
$\sigma$s above belong to one of them and $\sigma'$s to the other, just by reversing the role
of $i$ and $j$ in the permutation. This way
we obtain that the contribution of any point in the definition of $\lambda_\gamma$
is canceled by the contribution of the corresponding point from the other group and
so $\lambda_\gamma(V_1\times  \ldots\times V_{n+1})=0$ when one pair of points
$\{x_{\gamma+1},  \ldots, x_{\gamma+(n+2)}\}$ is not separated by any of the sets $V_1,  \ldots, V_{n+1}$.
But as previously noted this holds for all but finitely many $\gamma$s.

Now let us note that
for each $f \in C(K^{n+1})$ we have $\left(\lambda_{\gamma}(f)\right)_{\gamma \in \mathcal{L}(\omega_1)}\in c_0(\mathcal{L}(\omega_1)).$
Indeed, consider the operator $T: C(K^{n+1})\rightarrow\ell_\infty(\mathcal{L}(\omega_1))$ 
given by $T(f)=(\lambda_\gamma(f))_{\gamma<\omega_1}$.
It is well defined linear bounded operator  such that $T(f)\in c_0(\mathcal{L}(\omega_1))$
for any characteristic function of clopen subset of $K^{n+1}$.
So we obtain that it is true for any $f \in C(K^{n+1})$   by  applying the Weierstrass-Stone theorem
and the fact that $c_0(\mathcal{L}(\omega_1))$ is a closed subspace of $\ell_\infty(\mathcal{L}(\omega_1))$.

Now we are in position to construct the required projection on a copy of $c_0(\omega_1)$.
Since all points $x_{\gamma+i}$ for $\gamma \in \mathcal{L}(\omega_1)$ and $i\in \N$ are isolated,
we may consider $$Y=\overline{\mathrm{span}\{\chi_{\{(x_{\gamma+1},  \ldots, x_{\gamma+n+1})\}}:
\gamma \in \mathcal{L}(\omega_1)\}}$$ where the closure 
is taken in $C(K^{n+1})$. It is clear that $Y$ is isometric with $c_0(\omega_1)$.
We define 
$$P(f)=\sum_{\gamma \in \mathcal{L}(\omega_1)}\lambda_{\gamma}(f)\cdot 
\chi_{\{(x_{\gamma+1},\ldots, x_{\gamma+n+1})\}},\ f \in C(K^{n+1}).$$
For any $f \in C(K^{n+1})$ we have that $\left(\lambda_{\gamma}(f)\right)_{\gamma \in \mathcal{L}(\omega_1)}\in c_0(\mathcal{L}(\omega_1))$ and it follows that $P$ defines a bounded operator from $C(K^{n+1})$ 
to $C(K^{n+1})$. Moreover $P[C(K^{n+1})]\subseteq Y$.

To see that $P$ is a projection we only need to check that $P\restriction_{_{Y}}=Y$ and for that it is enough to prove that $P(\chi_{\{(x_{\gamma+1},  \ldots, x_{\gamma+n+1})\}})=\chi_{\{(x_{\gamma+1},  \ldots, x_{\gamma+n+1})\}}$ for each $\gamma \in \mathcal{L}(\omega_1)$. 
This is clear since for each $\gamma,\ \theta \in \mathcal{L}(\omega_1)$ the following holds:
\begin{displaymath}
\lambda_{\theta}(\chi_{\{(x_{\gamma+1},\ldots,x_{\gamma+n+1})\}})=\lambda_{\theta}(\{(x_{\gamma+1},  \ldots, x_{\gamma+n+1})\})= \left\{
\begin{array}{ll}
0 & \text{ if }\theta \neq \gamma;\\
1 & \text{ if }\theta = \gamma.
\end{array} \right.
\end{displaymath}
\end{proof}

\begin{corollary} Suppose $K$ is the space obtained under the assumption
$\clubsuit$ in Section 4 of \cite{KosZie}.
$C(K)$ is half-pcc and contains an isomorphic copy of $c_0(\omega_1)$.
$C(K\times K)$ contains a complemented copy of $c_0(\omega_1)$, in particular $C(K)$
is not weakly pcc.
\end{corollary}

\begin{proposition}\label{diversepcc}
If a compact scattered Hausdorff $K$ is $(n+1)$-diverse for some $n\in \N$, then $C(K^n)$ is half-pcc.
\end{proposition}
\begin{proof} 
Consider half-spaces $H_\alpha=\{f\in C(K^n): \int fd\mu_\alpha>a_\alpha\}$
for $a_\alpha\in \R$ and some Radon measures $\mu_\alpha$
on $K^n$ and any $\alpha<\omega_1$. We will prove that this
collection cannot be point-finite.  It is enough to prove that its refinement
is not point finite, so we may assume 
 that $a_\alpha$s are all equal to some $a\in\R$ bigger than uncountably many $a_\alpha$s. 
By going to an uncountable subset we may assume that there is an $\varepsilon>0$
and finite sets 
$G_\alpha\subseteq K$ such that there is $y_\alpha\in G_\alpha^n$
with 
$$|\mu_\alpha(\{y_\alpha\})|>2\varepsilon$$
 and
 $$|\mu_\alpha(K^n\setminus G_\alpha^n)|<\varepsilon.$$
Here we used the fact that all Radon measures on scattered compacta are atomic.
Going further to a smaller uncountable subsequence we may assume
that there is a partition $I_1\cup  \ldots\cup I_k=\{1, \ldots,n\}$ for some $k\leq n$
such that $y^\alpha_i=y^\alpha_j$ for all $\alpha<\omega_1$ if and only if $i, j\leq n$ are in the same set of
the partition.

We may assume that all the sets $G_\alpha=\{x_1^\alpha, \ldots, x_m^\alpha\}$
 have the same cardinality $m\in \N$ and
that the first $k$ of them in the above enumeration are all the $k\leq n$ coordinates of the point $y_\alpha$
 for all $\alpha<\omega_1$ so that $x_i^\alpha=y^\alpha_j$ if and only if $j\in I_i$
for $i\leq k$ and $j\leq n$.
Form points $(x_1^\alpha, \ldots, x_m^\alpha)\in K^m$.
They are $(\{1\}, ..., \{k\}, \{k+1, ..., m\})$-diverse  so let $(x_1, \ldots, x_m)$
be an $(\{1\}, ..., \{k\}, \{k+1, ..., m\})$-diverse cluster point. It follows that there are    pairwise  disjoint clopen neighbourhoods
$U_1, \ldots, U_k$ of $x_1, \ldots, x_k$ respectively such that 
$$\{x_{k+1},  \ldots, x_m\}\cap (U_1\cup  \ldots\cup U_k)=\emptyset.$$
Now let $V_j=U_i$ if and only if $j\in I_i$ for $j\leq n$ and $i\leq k$. Consider $V_1\times  \ldots\times V_n$.
We will prove now that for infinitely many $\alpha$s we have
$G_\alpha^n\cap V_1\times \ldots\times V_n=\{y_\alpha\}$. Indeed 
$W=U_1\times  \ldots\times U_k\times [K\setminus(U_1\cup \ldots \cup U_k)]^{m-k}$ is a clopen 
neighbourhood of the cluster point $(x_1, \ldots, x_m)$, so the
points $(x_1^\alpha,  \ldots, x_m^\alpha)$ are there for infinitely many $\alpha$s. 
 For the same $\alpha$s we must have
$y_\alpha\in  V_1\times \ldots\times V_n$.
Now if $y=(y_1, \ldots, y_n)\in G_\alpha^n$ is  in $V_1\times \ldots\times V_n$
note that none of the coordinates of $y$ may be among $\{x^\alpha_{k+1}, \ldots, x_m^\alpha\}$
because these must be in $K\setminus(U_1\cup  \ldots\cup U_k)=K\setminus(V_1\cup \ldots\cup V_n)$.
Only one of the remaining possible values $\{x_1^\alpha, \ldots, x^\alpha_k\}$ of the coordinates  of $y$
 may belong to $V_j$ for $j\leq n$, namely $x_i^\alpha$ such that $j\in I_i$, because 
 $V_j$s are
pairwise disjoint, so we conclude that  $y=y_\alpha$.
Hence  for infinitely many $\alpha<\omega_1$ we have
$$|\int \chi_{V_1\times \ldots\times V_n}d\mu_\alpha|\geq |\mu_\alpha(\{y_\alpha\})|-
|\int_{K^n\setminus G_\alpha^n}\chi_{V_1\times \ldots\times V_n}d\mu_\alpha|> \varepsilon.$$
Now considering $f=\pm({{a}\over {\varepsilon}})\chi_{V_1\times  \ldots\times V_n}$ we obtain
$\int fd\mu_\alpha>a$ for infinitely any $\alpha$s which shows that $H_\alpha$s do not form 
a point-finite family and completes the proof of the proposition.
\end{proof}


\section{A compact space from $\clubsuit$}
\label{sec:compact}

Our main result of this section is as follows:

\begin{proposition}[$\clubsuit$]\label{compactclub} 
For each $n\in \mathbb{N}$ there is a $(n+1)$-diverse non-separable Hausdorff compact scattered
topology $\tau$ on  $[0,\omega_1]$ of height $\omega+1$ and weight $\omega_1$
where sets $[0,\alpha+(n-1)]\cup\{\omega_1\}$ are closed for all $\alpha<\omega_1$. Moreover, there is a finite-to-one function $\phi:[0,\omega_1)\rightarrow [0,\omega_1)$ 
which is $\tau$ to the order topology continuous such that $|\phi^{-1}[\{\alpha\}]|\leq n$ for all $\alpha<\omega_1$.
\end{proposition}

The Ostaszewski's principle $\clubsuit$
(\cite{Osta}) is stated as follows:

\begin{definition}\label{ostaszewski} $\clubsuit$ is the following sentence: There is a sequence $({S}_{\alpha})_{\alpha \in \mathcal{L}(\omega_1)}$ such that for each $\alpha \in \mathcal{L}(\omega_1)$:
\begin{enumerate}
\item[(1)] ${S}_{\alpha}\subseteq \alpha$;
\item[(2)] ${S}_{\alpha}$ converges to $\alpha$ in the order topology;
\item[(3)] for every uncountable $X \subseteq \omega_1$ there is $\alpha \in \mathcal{L}(\omega_1)$ such that $ S_{\alpha}\subseteq X$.
\end{enumerate}
\end{definition}

In order to establish Theorem \ref{compactclub}, we need to enrich our terminology. Following the notation of \cite{KosZie}, for an ordinal $\alpha\leq \omega_1$ we put $F_0(\alpha)=\alpha=\{\beta:\beta<\alpha\}$ and for $n>0$ we let $F_{n+1}(\alpha)$ be the set of all finite sequences of elements of $F_n(\alpha)$. Define $F(\alpha)=\bigcup_{n\in \mathbb{N}}F_n(\alpha)$.  For $A \in F(\alpha)$ such that $A \in F_n(\alpha)$, by induction on $n \in \mathbb{N}$ we define the support of $A$, denoted $\mathrm{supp}(A)$, as the union off all sets $\mathrm{supp}(B)$ where $B$ is a term of the sequence $A$ with $\mathrm{supp}(A)=\{A\}$ for $A \in F_0(\alpha)$. If $A,\ B \in F(\alpha), \alpha < \omega_1$, then we say that $A < B$ if and only if $\beta <\gamma$ for every $\beta \in  \mathrm{supp}(A)$ and $\gamma \in  \mathrm{supp}(B)$. 

For a collection $\mathcal{W}$ of subsets of $[0,\omega_1]$, we say that is consecutive if and only if $A < B$ or $B < A$ whenever $A$ and $B$ are two distinct elements of $\mathcal{W}$. 

\begin{definition}
\label{convergence}
Given a collection $\mathcal{W}$ of subsets of $[0,\omega_1]$, we say that it converges to $\gamma \in \mathcal{L}(\omega_1)$ if and only if is consecutive and for every $\beta<\gamma$ the set $\left\{A \in \mathcal{W}: A\subsetneq [\beta+1,\gamma)\ \text{ is finite}\right\}$.
\end{definition}

For our purpose, we need the following version, which is in fact equivalent (\cite[Lemma 4.4]{KosZie}), of Ostaszewski's $\clubsuit$:

\begin{definition} $\clubsuit'$ is the following sentence: There is a sequence $(\mathcal{S}'_{\alpha})_{\alpha \in \mathcal{L}(\omega_1)}$ such that for each $\alpha \in \mathcal{L}(\omega_1)$:
\begin{enumerate}
\item[(1)] $\mathcal{S}'_{\alpha}\subseteq F(\alpha);$
\item[(2)] $\mathcal{S}'_{\alpha}$ converges to $\alpha$ in the sense of Definition \ref{convergence};
\item[(3)] for every uncountable consecutive $\mathcal{W} \subseteq F(\omega_1)$ there is $\alpha \in \mathcal{L}(\omega_1)$ such that $\mathcal{S}'_{\alpha}\subseteq \mathcal{W}$.
\end{enumerate}
\end{definition}

\begin{proof}[Proof of Theorem \ref{compactclub}]
Given $n \in \mathbb{N}$, we fix a sequence $(\mathcal{S}'_{\gamma})_{\gamma \in \mathcal{L}(\omega_1)}$ satisfying the conditions of $\clubsuit'$. This sequence hereafter will be denominated $\clubsuit'$-sequence. For each $\gamma \in \mathcal{L}(\omega_1)$, since $\mathcal{S}'_{\gamma}$ converges to $\gamma$ in the sense of (\ref{convergence}), without loosing generality we may assume that $\mathcal{S}'_{\gamma}=\{s_r(\gamma):r<\omega\}$ satisfies $\max\{\mathrm{supp}(s_r(\gamma))\}+(n-1)<\min\{\mathrm{supp}(s_t(\gamma))\}$ whenever $r<t<\omega$. 

The set of points of our space will be the set $[0,\omega_1)$ where for each $\gamma<\omega_1$ 
we will construct a countable neighborhood basis $\mathcal{B}_{\gamma}$. We perform this construction by transfinite induction. We start fixing on $[0,0]$ the topology $\tau_{0}$ generated by $\mathcal{B}_{0}=\{\{0\}\}$. Given $\gamma<\omega_1$, we assume that we have obtained for each $\beta<\gamma$ a topology $\tau_{\beta}$ on $[0,\beta]$ satisfying:
\begin{itemize}
\item[(1)] $[0,\beta]$ is a scattered locally compact Hausdorff space of height not bigger that $\omega$;
\item[(2)] each point $\delta\in [0,\beta]$ admits a countable neighborhood basis $\mathcal{B}_{\delta}$ consisting only of compact clopen sets and such that $\delta=\max\{V\}$ for each $V\in\mathcal{B}_{\delta}$;
\item[(3)] if $\alpha<\beta$, then $\tau_{\alpha}\subseteq \tau_{\beta}$ and $\{U\cap [0,\alpha]: U\in \tau_{\beta}\}=\tau_{\alpha}$;
\item[(4)] whenever $\alpha+(n-1)<\beta$, $[0,\alpha+(n-1)]$ is closed in $[0,\beta]$.
\end{itemize}
Then we consider   the topology $\tau^*_{\gamma}$ on $[0,\gamma)$ whose basis is $\bigcup_{\beta<\gamma}\tau_{\beta}.$

By condition (1)-(3) the space $[0,\gamma)$ endowed with the topology $\tau^*_{\gamma}$  is Hausdorff, locally compact scattered with height not bigger than $\omega$. We will find an appropriate countable local basis for $\gamma$, $\mathcal{B}_{\gamma}$, and then define on $[0,\gamma]$ the topology $\tau_{\gamma}$ generated by $\tau^*_{\gamma}\cup \mathcal{B}_{\gamma}$. 

If $\gamma$ is a successor ordinal, let $\gamma'$ be greatest limit ordinal smaller than $\gamma$. If $\gamma'+(n-1)<\gamma$ the we fix $\mathcal{B}_{\gamma}=\{\{\gamma\}\}$. If $\gamma'<\gamma\leq\gamma'+(n-1)$ then the collection $\mathcal{B}_{\gamma}$ will be determined by $\gamma'$ as we will see below. We then may assume that $\gamma$ is a limit ordinal and consider the set $\mathcal{S}'_{\gamma}=\{s_r(\gamma):r<\omega\}$ from our fixed $\clubsuit'$-sequence. If $\mathcal{S}'_{\gamma}$ cannot be represented as a collection consisting only of $(k+1)$-tuples, $k\leq n$, from $[\omega_1]^{n_1}\times\ldots\times [\omega_1]^{n_{k+1}}$ for some $\{n_1,\ldots,n_{k+1}\}\subseteq \mathbb{N}$, then we fix $\mathcal{B}_{\gamma+(i-1)}=\{\{\gamma+(i-1)\}\}$ for all $1\leq i \leq n$. Otherwise we write $s_r(\gamma)=\{G^1_r(\gamma),\ldots,G^{k+1}_r(\gamma)\}$ ($G^i_r(\gamma)\in [\omega_1]^{n_i}$, $1\leq i \leq k+1$) or, by abuse of notation, simply $s_r(\gamma)=\{G^1_r,\ldots,G^{k+1}_r\}$.
If for some $r<\omega$, the sets $G^1_r,\ldots,G^{k+1}_r$ are not pairwise disjoint or the height of the points of $G^1_r\cup \ldots\cup G^{k+1}_r$ is not uniformly bounded in $\tau^*_{\gamma}$ by some $p<\omega$, then we again fix $\mathcal{B}_{\gamma+(i-1)}=\{\{\gamma+(i-1)\}\}$ for all $1\leq i \leq n$. 

Otherwise we write $G^1_r\cup\ldots\cup G^{k+1}_r=\{x^1_r,\ldots,x^m_r\}$ and fix a partition $F_1,\ldots,F_{k+1}$ of $\{1,\ldots, m\}$, $|F_i|=n_i$  such that $G^i_r=\{x^s_r: s \in F_i\}$ for all $1\leq i \leq k+1$.

Now we use our our inductive hypotheses. By (4), for each $1\leq j \leq m$ and $r \geq 1$ the set $[\max\{\mathrm{supp}(s_{r-1}(\gamma))\}+n,x^j_{r}]$ is open in $\tau_{x^j_r}$ and then is also open in $\tau^*_{\gamma}$. Moreover by our construction these sets are also non-empty. By applying (1), (2) and (3), for each $1\leq i \leq k$ and for each $j \in F_i$, we may fix a collection of $\tau^*_{\gamma}$-clopen compact sets $\{W^j_r(\gamma):r\in \mathbb{N}\}$ so that:  
\begin{itemize}
\item  $x^j_{r}=\max\{W_r^j(\gamma)\}$;
\item  the heights in $\tau^*_{\gamma}$ of all points of $W_r^j(\gamma)$ are not bigger than $p$;
\item  $W^j_r(\gamma)\subseteq [\max(\mathrm{supp}(s_{r-1}(\gamma))+n,x^j_{r}]$;
\item  $(W^j_r(\gamma)\setminus \{x^j_{r}\})\cap (G^1_r\cup \ldots \cup G^{k+1}_r)=\emptyset$; 
\item  $W^j_r(\gamma)\cap W^l_r(\gamma)=\emptyset$ whenever $j\neq l$. 
\end{itemize}
For each $1\leq i \leq k$, and for each $r \in \mathbb{N}$ define: 
\begin{align*}
V_{r}(\gamma+(i-1))=\{\gamma+(i-1)\}\cup\left(\bigcup_{t\geq r} \bigcup_{l\in F_i} W^l_{t}(\gamma)\right).
\end{align*}

By construction it follows that for every $r,s\in \mathbb{N}$:
\begin{itemize}
\item $V_r(\gamma+(i-1))\cap V_s(\gamma+(j-1))=\emptyset$ whenever $i \neq j$; 
\item $V_r(\gamma+(i-1))\cap (G^j_s\cup G^{k+1}_s)=\emptyset$ whenever $i\neq j$. 
\end{itemize}
Moreover, if $\alpha+(n-1)<\gamma$, then for any $1\leq i \leq k$ there is $r\in \mathbb{N}$ such that: 
\begin{itemize}
\item $V_r(\gamma+(i-1))\subseteq [\alpha+n,\gamma+(i-1)]$. 
\end{itemize}

For each $1\leq i \leq k$ we fix $\mathcal{B}_{\gamma+(i-1)}=\{V_r(\gamma+(i-1)):r\in \mathbb{N}\}$ and for each $k+1\leq i \leq n$ we fix $\mathcal{B}_{\gamma+(i-1)}=\{\{\gamma+(i-1)\}\}$. It is standard to check that $([0,\gamma+(i-1)],\tau_{\gamma+(i-1)})$ satisfies conditions (1)-(3) for each $1\leq i \leq n$. Given $1\leq i \leq n$ we fix $\beta=\gamma+(i-1)$ and we show that $([0,\beta],\tau_{\beta})$ also satisfies (4). Assume $\alpha+(n-1)<\beta$ and by contradiction that there is $\delta \in \overline{[0,\alpha+(n-1)]}^{\tau_{\beta}}$ such that $\alpha+(n-1)<\delta$. We then consider $\mathcal{B}_{\delta}=\{V_r(\delta):r\in \mathbb{N}\}$. Since $\delta$ is not isolated in $([0,\beta],\tau_{\beta})$ we may assume that $\alpha+n<\delta$. Recalling the construction of our space, there must exist $r_0$ such that $\delta \in V_{r_0}(\delta)\subseteq [\alpha+n,\delta]$. This is a contradiction and we deduce that $[0,\alpha+(n-1)]$ is closed in $([0,\beta],\tau_{\beta})$. 

Finally, we consider on $[0,\omega_1)$ the locally compact Hausdorff topology generated by the basis $\bigcup_{\gamma<\omega_1}\tau_{\gamma}$. This space will be denoted by $L_{n}$ and its one-point compactification will be denoted by $K_{n}=L_{n} \stackrel{.}{\cup} \{\omega_1\}$. This concludes the construction of the compact space.

\begin{claim}[1]
There is a finite-to-one continuous function $\phi:L_{n}\rightarrow [0,\omega_1)$ such that $|\phi^{-1}[\{\alpha\}]|\leq n$ for all $\alpha<\omega_1$.
\end{claim}
Let $L_{n}^{(1)}$ be the set off all accumulation points of $L_{n}$. We define $\phi:L_{n}\to [0,\omega_1)$ by setting first $\phi(\alpha) = \alpha$ if $\alpha\in L_{n}\setminus L_{n}^{(1)}$. If $\alpha \in L_{n}^{(1)}$, then there is $\gamma \in \mathcal{L}(\omega_1)$ and $0\leq i \leq n-1$ such that $\alpha=\gamma+i$. In this case we define $\phi(\alpha)=\phi(\gamma+i)=\gamma$.

It is clear that $|\phi^{-1}[\{\alpha\}]|\leq n$ for all $\alpha<\omega_1$ and that $\phi[L_{n}]$ is homeomorphic to the interval $[0,\omega_1)$ with the order topology. To see that $\phi$ is continuous, let $\alpha=\gamma+i\in L_{n}^{(1)}$ where $\gamma \in \mathcal{L}(\omega_1)$ and $0\leq i \leq n-1$. Since $L_{n}$ is first countable, there is a sequence $(\xi_r)_{r\in \mathbb{N}}$ converging to $\alpha$. By the construction of the space, there is no loss of generality if we assume that $\xi_r<\alpha$ for all $r \in \mathbb{N}$. Since $\phi(\alpha)=\gamma\in \mathcal{L}(\omega_1)$, for any $\beta<\gamma$ we may fix the clopen set $[\beta,\gamma]$ in $[0,\omega_1)$ endowed with the order topology. By the construction of $L_{n}$,  the set $[\beta+n,\gamma+i]$ is open in $L_{n}$ and since $\xi_r$ converges to $\alpha$, there exist $r_0$ such that $\xi_r\in [\beta+n,\gamma+i]$ whenever $r \geq r_0$. From the definition of the function $\phi$, it follows that $\phi(\xi_r)\in [\beta,\gamma]$ whenever $r \geq r_0$.

\begin{claim}[2]
The compact $K_{n}$ is $(n+1)$-diverse. 
\end{claim}
Let $m\in \N$ and $F_1,\ldots,F_{k+1}$ a partition of $\{1,\ldots,m\}$ such that $k \leq n$. Let $\{(x_1^{\alpha},\ldots,x^{\alpha}_{m})\}_{\alpha<\omega_1}\subseteq K_{n}^{m}$ be a sequence of $(F_1,\ldots,F_{k+1})$-diverse points. We will prove that this sequence admits a $(F_1,\ldots,F_{k+1})$-diverse cluster point.

First of all, by reducing to an uncountable subset we may assume that the collection $\{\{x_1^{\alpha},\ldots,x^{\alpha}_{m}\}:\alpha<\omega_1\}$ constitutes a $\Delta$-system with root $\Delta$. Moreover, we may also assume that there exists $D \subseteq \{1,\ldots,m\}$ such that $\Delta=\{x^{\alpha}_{i}:i\in D\}$ and we will denote the element $x^{\alpha}_{i} \in \Delta$ by $a_{i}$. For each $1\leq j \leq k+1$ define $F'_j=F_j\setminus D$. If $F'_1\cup \ldots \cup F'_{k+1}=\emptyset$ then we are done. Without loss of generality we may assume that $F'_j\neq \emptyset$ for some $1\leq j \leq k+1$. 

By reducing to another uncountable subset if necessary, we may assume that there is at most only one $1\leq j \leq k+1$ such that if $x_i^{\alpha}=\omega_1$ for some $\alpha$, then $i\in F_{j}\cap D$ , moreover, without loss of generality we may suppose that such $j=k+1$ and then define $G^{k+1}_{\alpha}=\{x^{\alpha}_{i}:i\in F'_{k+1}\}$ for every $\alpha<\omega_1$. We also define $G_{\alpha}^{j}=\{x^{\alpha}_{s}:s\in F'_j\}$ for each $1\leq j \leq k$ and $\alpha<\omega_1$. 

By reenumerating terms we may assume that $G^{\alpha}_{1}\cup\ldots\cup G^{\alpha}_{k+1}=\{y_1^{\alpha},\ldots,y_p^{\alpha}\}$ for each $\alpha<\omega_1$, and it follows that  $\{(y_1^{\alpha},\ldots,y^{\alpha}_{p})\}_{\alpha<\omega_1}\subseteq K_n^{p}$ is a sequence of $(F'_1,\ldots,F'_{k+1})$-diverse points. 

We may fix the sequence $\mathcal{S}=\{\{G^{\alpha}_{1},\ldots, G^{\alpha}_{k+1}\}:\alpha<\omega_1\}$ and assume that is consecutive. Moreover, by going to an uncountable subset if necessary, we may assume that $(\Delta\setminus\{\omega_1\})<G^{\alpha}_{1}\cup\ldots\cup G^{\alpha}_{k+1}$ for all $\alpha$.

Recalling our $\clubsuit'$-sequence in the construction of $K$, there is exist $\gamma \in \mathcal{L}(\omega_1)$ such that $\mathcal{S}_{\gamma}'\subseteq \mathcal{S}$ and $\mathcal{S}_{\gamma}'$ can be enumerated in the increasing order as $$\mathcal{S}'_{\gamma}=\{\{G_1^{r}(\gamma),\ldots,G_{k+1}^{r}(\gamma)\}:r<\omega\}.$$  

We define $z=(z_1,\ldots,z_m)$ in the following way: for $1\leq i \leq k$, if $j \in F'_{i}=F_{i}\setminus D$ we set $z_j=\gamma+(i-1)$, if $j \in F'_{k+1}=F_{k+1}\setminus D$ then we put $z_i=\omega_1$ and if $j \in  D$ then we fix $z_j=a_j\in \Delta$. According to the construction of $K_{n}$, if $1\leq i \leq k$ then any sequence $(u_{r})_{r\in \mathbb{N}}$ such that $u_{r}\in G_i^{r}(\gamma)$ for every $r \in \mathbb{N}$ will converge to $\gamma+(i-1)$ and if $i=k+1$ will converge to $\omega_1$. It follows that $z$ is an accumulation point of the original sequence.  

If $z$ is not $(F_1,\ldots,F_{k+1})$-diverse, then there is $x\in \{z_j:j \in F_{i_1}\}\cap \{z_j:j \in F_{i_2}\}$ for some $i_1\neq i_2$. By definition of $z$, there are $j_1 \in F_{i_1}\cap D$ and $j_2 \in F_{i_1}\cap D$ such that $x=a_{j_1}=a_{j_2}\in \Delta$.  Since our original sequence is $(F_1,\ldots,F_{k+1})$-diverse, it follows that $\{a_j:j \in F_{i_1}\cap D\}\cap \{a_j:j \in F_{i_2}\cap D\}=\emptyset$, a contradiction. 

\begin{claim}[3]
The space $K_{n}$ has height $\omega+1$. 
\end{claim}

We observe that the height of $K_{n}$ cannot be bigger $\omega+1$. Indeed, for any $\gamma<\omega_1$, the height of the points of $V_0(\gamma)\setminus \{\gamma\}$ are uniformly bounded by some $p<\omega$, then the height of $\gamma$ cannot be bigger than $p + 1$. 

On the other hand, the height cannot be finite because
then the only accumulation points of sequences from the last Cantor-Bendixson level
would have to have all coordinates equal to $\omega_1$ which contradicts the $2$-diversity.
\end{proof}


\begin{proposition}[$\clubsuit$]\label{komega} 
There is a non-separable Hausdorff compact scattered
topology $\tau$ on  $[0,\omega_1]$  which is $n$-diverse for every $n$, has
height $\omega+1$ and weight $\omega_1$
where sets $[0,\alpha+(n-1)]\cup\{\omega_1\}$ are closed for all $\alpha<\omega_1$. 
Moreover, there is a finite-to-one function $\phi:[0,\omega_1)\rightarrow [0,\omega_1)$ 
which is $\tau$ to the order topology continuous.
\end{proposition}
\begin{proof} Do a similar construction as the previous one but allowing
the limit points $\gamma$ to split into $n$ points $\{\gamma+i:i<n\}$
without limiting $n\in \N$.
\end{proof}

\section{The spaces  of A. Dow, H. Junnila and J. Pelant}

In this section, the symbol $DJP_1$ stands for the compact Hausdorff space
constructed in \cite[Example 2.17, there denoted by $K_3$]{DowJunPel} under the assumption of $\diamondsuit$
and $DJP_2$ stands for a ZFC compact space from  \cite[Example 2.16, there denoted by $K_2$]{DowJunPel}.
It is proved in \cite{DowJunPel} that $C(DJP_1)$ is weakly pcc and admits a finite-to-one continuous
map onto $[0,\omega_1]$ with the order topology.  The proof
of the fact that $C(DJP_1)$ is weakly pcc in \cite{DowJunPel} relies on the fact
that it is pointwise pcc which means that every point-finite family of open sets
in the topology of pointwise convergence is at most countable. In fact, it is
proved in Lemma 2.13 of \cite{DowJunPel} that for $C(K)$ for  a scattered compact space  $K$
to be weakly pcc is the same as to be pointwise pcc.
On the other hand
 Arkhangelskii and Tka\v cuk implicitly proved in \cite{Arh} (see Proposition 2.2. of \cite{DowJunPel})
that
a compact Hausdorff $K$ is pointwise pcc if and only if $K^n\setminus \Delta_n$ is 
$\omega_1$-compact for all $n\in \mathbb{N}$, where $\omega_1$-compact means that
every set of cardinality $\omega_1$ has  an accumulation point and
$\Delta_n=\{(x_1,  \ldots, x_n)\in K^n:  \  x_i=x_j\ \hbox{\rm for some}\  i\not=j\}$.
Our main observation concerning these notions is the following:

\begin{lemma}\label{pointwisesquare} Suppose $K$ is a compact space. If $C(K)$ is pointwise pcc, then
$C(K^2)$ is pointwise pcc as well.
\end{lemma}
\begin{proof}
Suppose $K$ is pointwise pcc. Then by 2.2. of \cite {DowJunPel} based on 2.7 of \cite{Arh} we have that
$K^n\setminus \Delta_n$ is $\omega_1$-compact for each $n\in \mathbb{N}$.  By the same
result we need to
prove that $(K^2)^n\setminus \Delta_n({K^2})$ is $\omega_1$-compact for every $n\in \mathbb{N}$, where
$$\Delta_n({K^2})=\{(x_1, x_2,  \ldots, x_{2n-1}, x_{2_n}): \exists i\not=j\ \ (x_{2i+1}, x_{2i+2})
=(x_{2j+1}, x_{2j+2})\}$$
So $X=(K^2)^n\setminus \Delta_n({K^2})$ is 
$$\{(x_1, x_2,  \ldots, x_{2n-1}, x_{2n}): \forall i\not=j\ \ (x_{2i+1}, x_{2i+2})
\not=(x_{2j+1}, x_{2j+2})\}$$

Let $\{x_\xi=(x_1^\xi,x_2^\xi,\ldots,x_{2n-1}^\xi,x_{2n}^\xi)): \xi<\omega_1\}$ be an uncountable subset of
$X$. By going to an uncountable  subset we may assume that $x_j^\xi=x_l^\xi$ if and
only if  $x_j^\eta=x_l^\eta$ for all $\xi<\eta<\omega_1$; 
$1\leq j, l\leq n$.  This way we obtain a partition of $\{1, 2,  \ldots, 2n\}$ into sets 
$(A_k)_{1\leq k\leq m}$  for some $m\leq 2n$ such that for every $\xi<\omega_1$ we have 
$x^\xi_j=x^\xi_l$ if and only if  $j$ and $l$ are in the same set of the partition. 
Since the points are in $X=(K^2)^n\setminus \Delta_n({K^2})$ for every
$0\leq i<j< n$ it is not true that at the same time ${2i+1, 2_j+1}$ are in the same 
part of the partition and ${2i+2, 2_j+2}$ are in the same part  of the partition.

Choose a representative
$j_k$ of each element of $A_k$ of the partition and form a point $v_\xi\in (x^\xi_{j_1},  \ldots, x^\xi_{j_m})$
for each $\xi<\omega_1$. Since all the coordinates of $v_\xi$s are distinct they  are in
$K^m\setminus \Delta_m$. So by the pointwise pcc property of $C(K)$, we conclude that
$(v_\xi)_{\xi<\omega_1}$ has a cluster point, say $(v_1,  \ldots, v_m)$ in $K^m\setminus \Delta_m$,
that is with all the coordinates different. 

Now define a point $x$ of $(K^2)^n$ by putting $v_k$ on all the coordinates from $A_k$.
Since for
$0\leq i<j< n$ it is not true that at the same time ${2i+1, 2_j+1}$ are in the same 
part of the partition and ${2i+2, 2_j+2}$ are in the same part  of the partition,
$x$ is in $X=(K^2)^n\setminus \Delta_n({K^2})$ and $x$ must be a cluster point of
the $x_\xi$s.

\end{proof}

\begin{corollary}\label{powerfullywpcc}
Let $K$ be compact scattered space. If $C(K)$ is weakly pcc, then
$C(K^n)$ is weakly pcc for all $n\in\N$. In particular $C(K^n)$ 
does not contain a complemented copy of $c_0(\omega_1)$ for any $n\in \N$.
\end{corollary}
\begin{proof}
If $K$  is compact scattered and $C(K)$ is weakly pcc, then by Lemma 2.13 of \cite{DowJunPel}, it is pointwise
pcc. By the above Lemma \ref{pointwisesquare} we conlude that $C(K^{2^n})$
is pointwise pcc for every $n\in\N$ and so weakly pcc for every $n\in \N$
again  by Lemma 2.13 of \cite{DowJunPel}.  As $K^n$ is homeomorphic to a closed
subset of $K^{2^n}$, $C(K^n)$ is a quotient Banach space of $C(K^{2^n})$ 
and so is weakly pcc as well by Lemma 1.8 of \cite{DowJunPel}.
As weakly pcc implies half-pcc,
by \ref{introequiv} we conclude that $C(K^n)$ does  not contain a
complemented copy of $c_0(\omega_1)$ for any $n\in\N$.

\end{proof}

\begin{corollary}[$\clubsuit$]\label{powerfullydjp}
 $C(DJP_1^n)$ is weakly pcc for any $n\in\N$ and $C(DJP_1)$ is not Lindel\"of in the weak topology.
In particular $C(DJP_1^n)$ does  not contain a
complemented copy of $c_0(\omega_1)$ for any $n\in\N$.
\end{corollary}
\begin{proof}
It is shown in Example 2.17 of \cite{DowJunPel} that $C(DJP_1)$ is weakly pcc.
The examination of the construction leads to the conclusion that the same can be achieved
under $\clubsuit$ rather than $\diamondsuit$.
As proved in Example 2.17 of \cite{DowJunPel} $DJP_1$ maps continuously onto
$[0, \omega_1]$ and so $C([0,\omega_1])$ is a closed subspace of
$C(DJP_1)$. It is well-known however that $C([0,\omega_1])$ is not 
Lindel\"of in the weak topology (consider an open cover
by the sets $V_\alpha=\{f\in C([0,\omega_1]): |f(\alpha)|>0\}$  for
$\alpha<\omega_1$ and by $\{f: |f(\omega_1)|<1\}$), so $C(DJP_1)$ cannot be Lindel\"of in the weak topology.
\end{proof}

\begin{proposition}\label{djp2}
The space $DJP_2$ has countable height  and contains a point $\infty$ such that $DJP_2\setminus\{\infty\}$
maps injectively and continuously onto a subset of $\R$.
\end{proposition}
\begin{proof} The sets $A_i$s of the example 2.15 of \cite{DowJunPel} 
have heights not bigger than $i\in N$, and the entire space is obtained as 
the one point compactification of $\bigcup_{i\in\N}A_i$, so the height of $DJP_2$ is
$\omega+1$.
The topology on $\bigcup_{i\in\N}A_i$ is a refinement of the topology inherited
from $\R$ and so the identity is the desired continuous map on $DJP_2\setminus\{\infty\}$.
\end{proof}

The following should be compared with the fact that our space $K_1$ or
the space of \cite{KosZie} are $2$-diverse and not weakly pcc by 
Proposition \ref{compactclub} and Theorem \ref{diverseintro}.

\begin{proposition}\label{uselessdjp2} Suppose that $K$ is compact scattered 
space which contains a point $\infty$ such that $K\setminus\{\infty\}$
maps injectively and continuously onto a subset of $\R$. 
If $K$ is $2$-diverse, then $C(K)$ is weakly pcc.
\end{proposition}
\begin{proof} Let $\phi:K\setminus\{\infty\}\rightarrow\R$ denote
the continuous injective map. By Lemma 2.13 and Proposition 2.2. of \cite{DowJunPel}
it is enough to prove that $K^n\setminus \Delta_n$ is $\omega_1$-compact for
every $n\in \N$. So let $(x^1_\xi, ..., x^n_\xi)$ be points of $K^n\setminus \Delta_n$ for $\xi<\omega_1$.
By going to a  smaller power we may assume that they have all coordinates different than $\infty$.
By going to an uncountable subset we may assume that there is $\varepsilon>0$ such that
$|\phi(x^i_\xi)-\phi(x^j_\xi)|\geq\varepsilon$ for every $\xi<\omega_1$ and any distinct $i, j\leq n$.
Whenever $(x_1,..., x_n)$ is an accumulation point of
$\{(x^1_\xi, ..., x^n_\xi):\xi<\omega_1\}$ in $K^n$ and $x_i, x_j\in K\setminus\{\infty\}$,
$(x_i, x_j)$ is an accumulation point of $\{(x^i_\xi, x^j_\xi):\xi<\omega_1\}$ in $(K\setminus\{\infty\})^2$
and so $|\phi(x_i)-\phi(x_j)|\geq\varepsilon$ because otherwise
$\{(y_1,y_2): |\phi(y_i)-\phi(y_j)|<\varepsilon\}$ is its open neighbourhood in 
 $(K\setminus\{\infty\})^2$ which separates it from the set.

Now $(\infty, x^1_\xi, ..., x^n_\xi)$s are $(\{1\}, \{2,...n+1\})$-diverse points, so
by the hypothesis they should have a $(\{1\}, \{2,...n+1\})$-diverse accumulation point
(see Definitions \ref{diversepoint} and \ref{diversespace}).
But such a point would give rise to an accumulation point of $\{(x^1_\xi, ..., x^n_\xi):\xi<\omega_1\}$ 
in $(K\setminus\{\infty\})^n$ which, as we noted, must have all coordinates different, and
so is in $K^n\setminus \Delta_n$ as required for the  weak pcc.
\end{proof}


\bibliographystyle{amsalpha}

\begin{thebibliography}{A}

\bibitem{argyros} S. Argyros, J. Castillo, A.  Granero, M. Jimenez, J.  Moreno, 
\emph{Complementation and embeddings of $c_0(I)$ in Banach spaces}.
Proc. London Math. Soc. (3) 85 (2002), no. 3, 742--768. 


\bibitem{Arh}  A. Arhangelʹskii, V. Tka\v cuk, \emph{Calibers and point-finite cellularity of the space 
$C_p(X)$ and some questions of S. Gulʹko and M. Husek}. Topology Appl. 23 (1986), no. 1, 65--73.

\bibitem{Cem} P. Cembranos, 
\emph{$C(K,E)$ contains a complemented copy of $c_0$}, 
Proc. Amer. Math. Soc. 91 (1984), 556--558. 

\bibitem{dowsimon} A. Dow, P.  Simon, \emph{Spaces of continuous functions over a $\Psi$-space}. 
Topology Appl. 153 (2006), no. 13, 2260--2271.

\bibitem{DowJunPel} A. Dow, H. Junnila, J. Pelant,
\emph{Chain condidition and weak topologies},  
Topology Appl. 156 (2009), 1327--1344.

\bibitem{ForMagShe} M. Foreman, M. Magidor, S. Shelah, 
\emph{Martin's maximum, saturated ideals, and nonregular ultrafilters},
Ann. Of Math.  127 (1988) 1--47.

\bibitem{freniche} F. Freniche, 
\emph{Barrelledness of the space of vector valued and simple functions}.
Math. Ann. 267 (1984), no. 4, 479--486. 

\bibitem{GaHa} E. Galego, J. Hagler,
\emph{Copies of $c_0(\varGamma)$ in $C(K,X)$ spaces}. 
Proc. Amer. Math. Soc. 140 (2012), 3843--3852. 



\bibitem{josefson} B. Josefson, 
\emph{Weak sequential convergence in the dual of a Banach space does not imply norm convergence}.
Bull. Amer. Math. Soc. 81 (1975), 166--168. 

\bibitem{KosZie} P. Koszmider, P. Zieli\'nski,
\emph{Complementations and decompositions in weakly Lindel\"{o}f Banach spaces },  
J. Math. Anal. Appl 376 (2011), 329--341.

\bibitem{Kunen} K. Kunen, 
\emph{Set Theory. An Introduction to Independence Proofs}, 
Stud. Logic Found. Math., vol. 102, North-Holland, Amsterdam, 1980.

\bibitem{nissenzweig} A. Nissenzweig,  \emph{$w^*$ sequential convergence}.
 Israel J. Math. 22 (1975), no. 3-4, 266--272. 

\bibitem{Osta} K. Ostaszewski, 
\emph{On countably compact, perfectly normal spaces}, 
J. Lond. Math. Soc. 14 (1976) 505--516.

\bibitem{pol} R. Pol, \emph{Concerning function spaces on separable compact spaces}. Bull. Acad. Polon. Sci. Ser. Sci. Math. Astronom. Phys. 25 (1977), no. 10, 993--997.


\bibitem{rosenthal} H.  Rosenthal, \emph{On injective Banach spaces and the spaces $L_\infty(\mu)$
 for finite measure $\mu$}. Acta Math. 124, 1970 , 205--248. 

\bibitem{ryan} R. Ryan,
\emph{Complemented copies of $c_0$ in spaces of compact operators}.
Proc. Roy. Irish Acad. Sect. A 91 (1991), no. 2, 239--241. 


\bibitem{saabsaab} E. Saab, P.  Saab,
\emph{On complemented copies of $c_0$ in injective tensor products}. 
Geometry of normed linear spaces (Urbana-Champaign, Ill., 1983), 131--135, 
Contemp. Math., 52, Amer. Math. Soc., Providence, RI, 1986.


\bibitem{Se}  Z. Semadeni,
\emph{Banach Spaces of Continuous Functions Vol. I},  
Monografie Matematyczne, Tom 55. Warsaw, PWN-Polish Scientinfic Publishers, Warsaw, 1971.

\bibitem{sokolov}  G. A. Sokolov, \emph{Lindel\"of property and the iterated 
continuous function spaces}. Fund. Math. 143 (1993), no. 1, 87--95.

\bibitem{Todo} S. Todor\v cevi\'c, 
\emph{Biorthogonal systems and quotient spaces via Baire category methods}, 
Math. Ann. 335 (2006), 687--715.

\end{thebibliography}

\end{document}